\documentclass[12pt]{article}
\usepackage{amsmath}\usepackage{epsf,amsfonts,amsthm}\usepackage{amscd,mathptmx,amssymb}
\usepackage{xcolor,epic,eepic}\usepackage{epsfig}
\usepackage{fontenc,indentfirst, delarray,amsfonts,amsmath,amssymb}
\usepackage{rotating}
\usepackage[T1]{fontenc}
\usepackage[matrix,arrow,curve]{xy}
\usepackage{amsmath}
\usepackage{amssymb}
\usepackage{amsthm}
\usepackage{amscd}
\usepackage{amsfonts}
\usepackage{graphicx}%
\usepackage{fancyhdr}
\usepackage{dsfont,texdraw}
\usepackage{amsmath}\usepackage{epsf,amsfonts,amsthm}\usepackage{amscd,mathptmx,amssymb}
\usepackage{xcolor,epic,eepic}\usepackage{epsfig}
\usepackage{fontenc,indentfirst, delarray,amsfonts,amsmath,amssymb}
\usepackage{rotating}
\usepackage[T1]{fontenc}
\newcommand{\bee}{\begin{enumerate}}
\newcommand{\eee}{\end{enumerate}}
\newcommand{\benn}{\begin{equation*}}
\newcommand{\eenn}{\end{equation*}}
\newcommand{\be}{\begin{equation}}
\newcommand{\ee}{\end{equation}}
\newcommand{\bean}{\begin{eqnarray}}
\newcommand{\eean}{\end{eqnarray}}
\newcommand{\bea}{\begin{eqnarray*}}
\newcommand{\eea}{\end{eqnarray*}}

\newcommand{\E}{\ell}
\newcommand{\N}{\mathbb{N}}

\newcommand{\K}{\mathbb{K}}

\newcommand{\op}[1]{\!\!\mathop{\rm ~#1}\nolimits}

\mathchardef\za="710B  
\mathchardef\zb="710C  
\mathchardef\zg="710D  
\mathchardef\zd="710E  
\mathchardef\zve="710F 
\mathchardef\zz="7110  
\mathchardef\zh="7111  
\mathchardef\zy="7112 

\mathchardef\zi="7113  
\mathchardef\zk="7114  
\mathchardef\zl="7115  
\mathchardef\zm="7116  
\mathchardef\zn="7117  
\mathchardef\zx="7118  
\mathchardef\zp="7119  
\mathchardef\zr="711A  
\mathchardef\zs="711B  
\mathchardef\zt="711C  
\mathchardef\zu="711D  
\mathchardef\zf="711E 
\mathchardef\zq="711F  
\mathchardef\zc="7120  
\mathchardef\zw="7121  
\mathchardef\ze="7122  
\mathchardef\zvy="7123  
\mathchardef\zvw="7124  
\mathchardef\zvr="7125 
\mathchardef\zvs="7126 
\mathchardef\zvf="7127  
\mathchardef\zG="7000  
\mathchardef\zD="7001  
\mathchardef\zY="7002  
\mathchardef\zL="7003  
\mathchardef\zX="7004  
\mathchardef\zP="7005  
\mathchardef\zS="7006  
\mathchardef\zU="7007  
\mathchardef\zF="7008  
\mathchardef\zW="700A  

\newcommand{\cyclic}{\mathop{\kern0.9ex{{+}
\kern-2.15ex\raise-.25ex\hbox{\Large\hbox{$\circlearrowright$}}}}\limits}

\newtheorem{rem}{Remark}
\newtheorem{theo}{Theorem}
\newtheorem{prop}{Proposition}
\newtheorem{lem}{Lemma}

\newtheorem{ex}{Example}
\newtheorem{defi}{Definition}

\begin{document}
\title{Higher categorified algebras versus bounded homotopy algebras}
\author{David Khudaverdyan, Ashis Mandal, and Norbert Poncin\footnote{University of
Luxembourg, Campus Kirchberg, Mathematics Research Unit, 6, rue
Richard Coudenhove-Kalergi, L-1359 Luxembourg City, Grand-Duchy of
Luxembourg. The research of D. Khudaverdyan and N. Poncin  was
supported by UL-grant SGQnp2008. A. Mandal thanks the
Luxembourgian NRF for support via AFR grant PDR-09-062.}}
\date{}
\maketitle

\begin{abstract} We define Lie 3-algebras and prove that these are
in 1-to-1 correspondence with the 3-term Lie infinity algebras
whose bilinear and trilinear maps vanish in degree $(1,1)$ and in
total degree 1, respectively. Further, we give an answer to a
question of \cite{Roy07} pertaining to the use of the nerve and
normalization functors in the study of the relationship between
categorified algebras and truncated sh algebras.
\end{abstract}
\maketitle \noindent{\bf{Mathematics Subject Classification (2000)
:}} 18D05, 55U15, 17B70, 18D10, 18G30. 

\section{Introduction}

Higher structures -- infinity algebras and other objects up to
homotopy, higher categories, ``oidified'' concepts, higher Lie
theory, higher gauge theory... -- are currently intensively
investigated. In particular, higher generalizations of Lie
algebras have been conceived under various names, e.g. Lie
infinity algebras, Lie $n$-algebras, quasi-free differential
graded commutative associative algebras ({\small qfDGCAs} for
short), $n$-ary Lie algebras, see e.g. \cite{Dzh05}, crossed
modules \cite{MP09} ...See also \cite{AP10}, \cite{GKP11}.\medskip

More precisely, there are essentially two ways to increase the
flexibility of an algebraic structure: homotopification and
categorification.\medskip

Homotopy, sh or infinity algebras \cite{Sta63} are homotopy
invariant extensions of differential graded algebras. This
property explains their origin in BRST of closed string field
theory. One of the prominent applications of Lie infinity algebras
\cite{LS93} is their appearance in Deformation Quantization of
Poisson manifolds. The deformation map can be extended from
differential graded Lie algebras ({\small DGLA}s) to
$L_{\infty}$-algebras and more precisely to a functor from the
category {\tt L}$_{\infty}$ to the category {\tt Set}. This
functor transforms a weak equivalence into a bijection. When
applied to the {\small DGLA}s of polyvector fields and
polydifferential operators, the latter result, combined with the
formality theorem, provides the 1-to-1 correspondence between
Poisson tensors and star products.\medskip

On the other hand, categorification \cite{CF94}, \cite{Cra95} is
characterized by the replacement of sets (resp. maps, equations)
by categories (resp. functors, natural isomorphisms). Rather than
considering two maps as equal, one details a way of identifying
them. Categorification is a sharpened viewpoint that leads to
astonishing results in TFT, bosonic string theory... Categorified
Lie algebras, i.e. Lie 2-algebras (alternatively, semistrict Lie
2-algebras) in the category theoretical sense, have been
introduced by J. Baez and A. Crans \cite{BC04}. Their
generalization, weak Lie 2-algebras (alternatively, Lie
2-algebras), has been studied by D. Roytenberg
\cite{Roy07}.\medskip

It has been shown in \cite{BC04} that categorification and
homotopification are tightly connected. To be exact, Lie
2-algebras and 2-term Lie infinity algebras form equivalent
2-categories. Due to this result, Lie $n$-algebras are often
defined as sh Lie algebras concentrated in the first $n$ degrees
\cite{Hen08}. However, this `definition' is merely a
terminological convention, see e.g. Definition 4 in
\cite{SS07Structure}. On the other hand, Lie infinity algebra
structures on an $\N$-graded vector space $V$ are in 1-to-1
correspondence with square 0 degree -1 (with respect to the
grading induced by $V$) coderivations of the free reduced graded
commutative associative coalgebra $S^c(s V)$, where $s$ denotes
the suspension operator, see e.g. \cite{SS07Structure} or
\cite{GK94}. In finite dimension, the latter result admits a
variant based on {\small qfDGCAs} instead of coalgebras. Higher
morphisms of free {\small DGCAs} have been investigated under the
name of derivation homotopies in \cite{SS07Structure}. Quite a
number of examples can be found in \cite{SS07Zoo}.\medskip

Besides the proof of the mentioned correspondence between Lie
2-algebras and 2-term Lie infinity algebras, the seminal work
\cite{BC04} provides a classification of all Lie infinity
algebras, whose only nontrivial terms are those of degree 0 and
$n-1$, by means of a Lie algebra, a representation and an
$(n+1)$-cohomology class; for a possible extension of this
classification, see \cite{Bae07}.

In this paper, we give an explicit categorical definition of Lie
3-algebras and prove that these are in 1-to-1 correspondence with
the 3-term Lie infinity algebras, whose bilinear and trilinear
maps vanish in degree $(1,1)$ and in total degree 1, respectively.
Note that a `3-term' Lie infinity algebra implemented by a
4-cocycle \cite{BC04} is an example of a Lie 3-algebra in the
sense of the present work.\medskip

The correspondence between categorified and bounded homotopy
algebras is expected to involve classical functors and chain maps,
like e.g. the normalization and Dold-Kan functors, the (lax and
oplax monoidal) Eilenberg-Zilber and Alexander-Whitney chain maps,
the nerve functor... We show that the challenge ultimately resides
in an incompatibility of the cartesian product of linear
$n$-categories with the monoidal structure of this category, thus
answering a question of \cite{Roy07}.\medskip

The paper is organized as follows. Section 2 contains all relevant
higher categorical definitions. In Section 3, we define Lie
3-algebras. The fourth section contains the proof of the mentioned
1-to-1 correspondence between categorified algebras and truncated
sh algebras -- the main result of this paper. A specific aspect of
the monoidal structure of the category of linear $n$-categories is
highlighted in Section 5. In the last section, we show that this
feature is an obstruction to the use of the Eilenberg-Zilber map
in the proof of the correspondence ``bracket functor -- chain
map''.

\section{Higher linear categories and bounded chain complexes of vector spaces}

Let us emphasize that notation and terminology used in the present
work originate in \cite{BC04}, \cite{Roy07}, as well as in
\cite{Lei04}. For instance, a linear $n$-category will be an (a
strict) $n$-category \cite{Lei04} in {\tt Vect}. Categories in
{\tt Vect} have been considered in \cite{BC04} and also called
internal categories or 2-vector spaces. In \cite{BC04}, see
Sections 2 and 3, the corresponding morphisms (resp. 2-morphisms)
are termed as linear functors (resp. linear natural transformations),
and the resulting 2-category is denoted by {\tt VectCat} and also
by {\tt 2Vect}. Therefore, the $(n+1)$-category made up by linear
$n$-categories ($n$-categories in {\tt Vect} or $(n+1)$-vector
spaces), linear $n$-functors... will be denoted by {\tt Vect}
$n$-{\tt Cat} or {\tt$(n+1)$Vect}.\medskip

The following result is known. We briefly explain it here as its
proof and the involved concepts are important for an easy reading
of this paper.

\begin{prop}\label{EquivCat} The categories {\tt Vect} $n$-{\tt Cat} of linear
$n$-categories and linear $n$-functors and {\tt C}$^{n+1}(${\tt Vect}$)$ of
$(n+1)$-term chain complexes of vector spaces and linear chain maps
are equivalent.\end{prop}

We first recall some definitions.

\begin{defi} An {\bf $\mathbf n$-globular vector space} $L$, $n\in \N$, is a
sequence
\be\label{GlobVs}L_n\stackrel{s,t}{\rightrightarrows}L_{n-1}\stackrel{s,t}{\rightrightarrows}\ldots\stackrel{s,t}{\rightrightarrows}L_0\rightrightarrows
0,\ee of vector spaces $L_m$ and linear maps $s,t$ such that \be
s(s(a))=s(t(a))\;\;\mbox{and}\;\;t(s(a))=t(t(a)),\label{GlobCond}\ee
for any $a\in L_m,$ $m\in\{1,\ldots,n\}$. The maps $s,t$ are
called {\bf source map} and {\bf target map}, respectively, and
any element of $L_m$ is an {\bf m-cell}.\end{defi}

By higher category we mean in this text a {\bf strict} higher
category. Roughly, a linear $n$-category, $n\in\N$, is an
$n$-globular vector space endowed with compositions of $m$-cells,
$0<m\le n$, along a $p$-cell, $0\le p<m$, and an identity
associated to any $m$-cell, $0\le m<n$. Two $m$-cells $(a,b)\in
L_m\times L_m$ are composable along a $p$-cell, if
$t^{m-p}(a)=s^{m-p}(b)$. The composite $m$-cell will be denoted by
$a\circ_p b$ (the cell that `acts' first is written on the left)
and the vector subspace of $L_m\times L_m$ made up by the pairs of
$m$-cells that can be composed along a $p$-cell will be denoted by
$L_m\times_{L_p}L_m$. The following figure schematizes the
composition of two 3-cells along a 0-, a 1-, and a
2-cell.
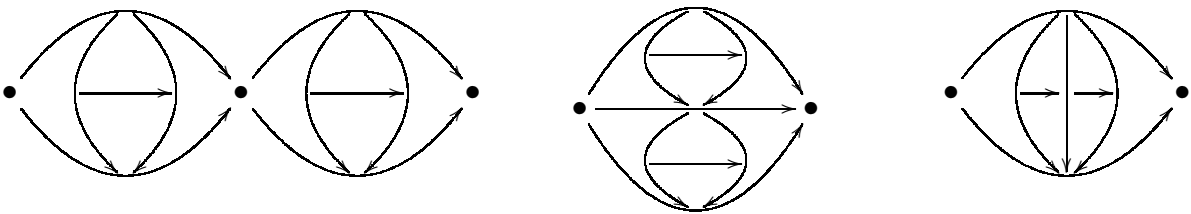
\vspace{5mm}
$$
\xymatrix@C=1.2pc@R=2pc{
&&\ar@/^1.6pc/@{->}[dd]\ar@/_1.6pc/@{->}[dd]&&&&\ar@/^1.6pc/@{->}[dd]\ar@/_1.6pc/@{->}[dd]&&\\
\bullet\ar@/^2.6pc/@{->}[rrrr]\ar@/_2.6pc/@{->}[rrrr]&\ar@{->}[rr]&&&\bullet\ar@/^2.6pc/@{->}[rrrr]\ar@/_2.6pc/@{->}[rrrr]&\ar@{->}[rr]&&&\bullet\\
&&&&&&}
\hspace{1cm}
\xymatrix@C=1.2pc@R=1pc{
&&\ar@/^1.6pc/@{->}[dd]\ar@/_1.6pc/@{->}[dd]&&\\
&\ar@{->}[rr]&&&\\
\bullet\ar@{->}[rrrr]\ar@/^3.2pc/@{->}[rrrr]\ar@/_3.2pc/@{->}[rrrr]&&\ar@/^1.6pc/@{->}[dd]\ar@/_1.6pc/@{->}[dd]&&\bullet\\
&\ar@{->}[rr]&&&\\
&&&&&&
}
\hspace{0cm}
\xymatrix@C=1.2pc@R=2pc{
&&\ar@{->}[dd]\ar@/^1.6pc/@{->}[dd]\ar@/_1.6pc/@{->}[dd]&&\\
\bullet\ar@/^2.6pc/@{->}[rrrr]\ar@/_2.6pc/@{->}[rrrr]&\ar@{->}[r]&\ar@{->}[r]&&\bullet\\
&&&&&&
}
$$

\begin{defi} A {\bf linear n-category}, $n\in\N$, is an $n$-globular vector space $L$ (with source and target maps $s,t$) together
with, for any $m\in\{1,\ldots,n\}$ and any $p\in\{0,\ldots,m-1\}$,
a linear composition map $\circ_p:L_m\times_{L_p}L_m\to L_m$ and,
for any $m\in\{0,\ldots,n-1\}$, a linear identity map $1:L_m\to
L_{m+1}$, such that the properties

\begin{itemize}\item for $(a,b)\in L_m\times_{L_p}L_m$, $$\mbox{if }\; p=m-1, \mbox{ then }\, s(a\circ_p b)=s(a) \mbox{ and }\; t(a\circ_pb)=t(b),$$
$$\mbox{if }\; p\le m-2, \mbox{ then }\, s(a\circ_p b)=s(a)\circ_ps(b) \mbox{ and }\; t(a\circ_pb)=t(a)\circ_pt(b),$$
\item $$s(1_a)=t(1_a)=a,$$ \item for any $(a,b),(b,c)\in
L_m\times_{L_p} L_m$, $$(a\circ_pb)\circ_p
c=a\circ_p(b\circ_pc),$$ \item
$$1^{m-p}_{s^{m-p}a}\circ_pa=a\circ_p1^{m-p}_{t^{m-p}a}=a$$
\end{itemize} are verified, as well as the compatibility
conditions \begin{itemize}\item for $q<p$, $(a,b),(c,d)\in
L_m\times_{L_p}L_m$ and $(a,c),(b,d)\in L_m\times_{L_q}L_m$,
$$(a\circ_pb)\circ_q(c\circ_pd)=(a\circ_qc)\circ_p(b\circ_qd),$$
\item for $m<n$ and $(a,b)\in L_m\times_{L_p}L_m$,
$$1_{a\circ_pb}=1_a\circ_p1_b.$$
\end{itemize}
\end{defi}

The morphisms between two linear $n$-categories are the linear
$n$-functors.

\begin{defi} A {\bf linear n-functor} $F:L\to L'$ between two linear
$n$-categories is made up by linear maps $F:L_m\to L'_m$,
$m\in\{0,\ldots,n\}$, such that the categorical structure --
source and target maps, composition maps, identity maps -- is
respected. \end{defi}

Linear $n$-categories and linear $n$-functors form a category {\tt
Vect} $n$-{\tt Cat}, see Proposition \ref{EquivCat}. To disambiguate this proposition, let us specify that the objects of {\tt C}$^{n+1}(${\tt Vect}$)$ are the complexes whose underlying vector space $V=\oplus_{i=0}^nV_i$ is made up by $n+1$ terms $V_i$.\medskip

The proof of Proposition \ref{EquivCat} is based upon the following result.

\begin{prop}\label{UniqueCatStr} Let $L$ be any $n$-globular vector space with linear identity maps. If $s_m$ denotes the restriction of the source map to $L_m$, the vector spaces
$L_m$ and $L'_m:=\oplus_{i=0}^mV_i$, $V_i:=\op{ker}s_i$,
$m\in\{0,\ldots,n\}$, are isomorphic. Further, the $n$-globular
vector space with identities can be completed in a unique way by
linear composition maps so to form a linear $n$-category. If we
identify $L_m$ with $L_m'$, this unique linear $n$-categorical
structure reads \be\label{source}
s(v_0,\ldots,v_m)=(v_{0},\ldots,v_{m-1}),\ee \be
\label{target}t(v_0,\ldots,v_m)=(v_0,\ldots,v_{m-1}+tv_m),\ee \be
\label{identity}1_{(v_0,\ldots,v_m)}=(v_{0},\ldots,v_m,0),\ee \be
\label{composition}(v_0,\ldots,v_m)\circ_p(v'_0,\ldots,v'_m)=(v_0,\ldots,v_p,v_{p+1}+v_{p+1}',\ldots,v_m+v_m'),\ee
where the two $m$-cells in Equation (\ref{composition}) are
assumed to be composable along a $p$-cell.\end{prop}

\begin{proof} As for the first part of this proposition, if $m=2$ e.g., it suffices to observe that the linear
maps
$$\za_L: L_2'=V_0\oplus V_1\oplus V_2\ni (v_0,v_1,v_2)\mapsto
1^2_{v_0}+1_{v_1}+v_2\in L_2$$ and $$\zb_L: L_2\ni a\mapsto
(s^2a,s(a-1^2_{s^2a}),a-1_{s(a-1^2_{s^2a})}-1^2_{s^2a})\in
V_0\oplus V_1\oplus V_2=L_2'$$ are inverses of each other. For arbitrary $m\in\{0,\ldots,n\}$ and $a\in L_m$, we set $$\zb_La=\left(s^ma,\ldots, s^{m-i}(a-\sum_{j=0}^{i-1}1^{m-j}_{p_{j}\zb_La}),\ldots,a-\sum_{j=0}^{m-1}1^{m-j}_{p_{j}\zb_La}\right)\in V_0\oplus\ldots\oplus V_i\oplus\ldots\oplus V_m= L'_m,$$ where $p_j$ denotes the projection $p_j:L_m'\to V_j$ and where the components must be computed from left to right.\medskip

For the second claim, note that when reading the source, target
and identity maps through the detailed isomorphism, we get
$s(v_0,\ldots,v_m)=(v_{0},\ldots,v_{m-1})$,
$t(v_0,\ldots,v_m)=(v_0,\ldots,v_{m-1}+tv_m)$, and
$1_{(v_0,\ldots,v_m)}=(v_{0},\ldots,v_m,0)$. Eventually, set
$v=(v_0,\ldots,v_m)$ and let $(v,w)$ and $(v',w')$ be two pairs of
$m$-cells that are composable along a $p$-cell. The composability
condition, say for $(v,w)$, reads
$$(w_0,\ldots,w_p)=(v_0,\ldots,v_{p-1},v_p+tv_{p+1}).$$ It follows
from the linearity of $\circ_p:L_m\times_{L_p}L_m\to L_m$ that
$(v+v')\circ_p(w+w')=(v\circ_pw)+(v'\circ_pw')$. When taking
$w=1^{m-p}_{t^{m-p}v}$ and $v'=1^{m-p}_{s^{m-p}w'}$, we find
$$(v_0+w'_0,\ldots,v_p+w'_p,v_{p+1},\ldots,v_m)\circ_p(v_0+w'_0,\ldots,v_p+w'_p+tv_{p+1},w'_{p+1},\ldots,w'_m)$$
$$=(v_0+w'_0,\ldots,v_m+w'_m),$$ so that $\circ_p$ is necessarily the composition given by Equation
(\ref{composition}). It is easily seen that, conversely, Equations
(\ref{source}) -- (\ref{composition}) define a linear
$n$-categorical structure. \end{proof}

\begin{proof}[Proof of Proposition \ref{EquivCat}] We define functors $\frak N:$ {\tt Vect} $n$-{\tt Cat} $\to $ {\tt C}$^{n+1}(${\tt Vect}$)$ and
$\frak G:$ {\tt C}$^{n+1}(${\tt Vect}$)$ $\to$ {\tt Vect} $n$-{\tt Cat} that
are inverses up to natural isomorphisms.\medskip

If we start from a linear $n$-category $L$, so in particular from
an $n$-globular vector space $L$, we define an $(n+1)$-term chain
complex ${\frak N}(L)$ by setting $V_m=\op{ker}s_m\subset L_m$ and
$d_m=t_m|_{V_m}:V_m\to V_{m-1}$. In view of the globular space
conditions (\ref{GlobCond}), the target space of $d_m$ is actually
$V_{m-1}$ and we have $d_{m-1}d_mv_m=0.$\medskip

Moreover, if $F:L\to L'$ denotes a linear $n$-functor, the value
${\frak N}(F):V\to V'$ is defined on $V_m\subset L_m$ by ${\frak
N}(F)_m=F_m|_{V_{m}}:V_m\to V'_{m}$. It is obvious that ${\frak
N}(F)$ is a linear chain map.\medskip

It is obvious that ${\frak N}$ respects the categorical structures
of {\tt Vect} $n$-{\tt Cat} and {\tt C}$^{n+1}$({\tt
Vect}).\medskip

As for the second functor $\frak G$, if $(V,d)$,
$V=\oplus_{i=0}^nV_i$, is an $(n+1)$-term chain complex of vector
spaces, we define a  linear $n$-category ${\frak G}(V)=L$,
$L_m=\oplus_{i=0}^mV_i$, as in Proposition \ref{UniqueCatStr}: the
source, target, identity and composition maps are defined by
Equations (\ref{source}) -- (\ref{composition}), except that
$tv_m$ in the {\small RHS} of Equation (\ref{target}) is replaced
by $dv_m$.\medskip

The definition of $\frak G$ on a linear chain map $\zf:V\to V'$
leads to a linear $n$-functor ${\frak G}(\zf):L\to L'$, which is
defined on $L_m=\oplus_{i=0}^mV_i$ by ${\frak
G}(\zf)_m=\oplus_{i=0}^m\zf_i$. Indeed, it is readily checked that
${\frak G}(\zf)$ respects the linear $n$-categorical structures of
$L$ and $L'$.\medskip

Furthermore, ${\frak G}$ respects the categorical structures of
{\tt C}$^{{n+1}}$({\tt Vect}) and {\tt Vect} $n$-{\tt
Cat}.\medskip

Eventually, there exist natural isomorphisms
$\za:\frak{NG}\Rightarrow\op{id}$ and
$\zg:\frak{GN}\Rightarrow\op{id}$.\medskip

To define a natural transformation
$\za:\frak{NG}\Rightarrow\op{id}$, note that $L'=(\frak{NG})(L)$
is the linear $n$-category made up by the vector spaces
$L'_m=\oplus_{i=0}^mV_i$, $V_i=\op{ker}s_i$, as well as by the
source, target, identities and compositions defined from $V={\frak
N}(L)$ as in the above definition of ${\frak G}(V)$, i.e. as in
Proposition \ref{UniqueCatStr}. It follows that $\za_L:L'\to L$,
defined by $\za_L:L_m'\ni(v_0,\ldots,v_m)\mapsto 1^m_{v_0}+\ldots
+1_{v_{m-1}}+v_m\in L_m$, $m\in\{0,\ldots,n\}$, which pulls the
linear $n$-categorical structure back from $L$ to $L'$, see
Proposition \ref{UniqueCatStr}, is an invertible linear
$n$-functor. Moreover $\za$ is natural in $L$.\medskip

It suffices now to observe that the composite $\frak{GN}$ is the
identity functor.\end{proof}

Next we further investigate the category {\tt Vect} $n$-{\tt Cat}.

\begin{prop} The category {\tt Vect} $n$-{\tt Cat} admits finite products.\end{prop}

Let $L$ and $L'$ be two linear $n$-categories. The product linear
$n$-category $L\times L'$ is defined by $(L\times L')_m=L_m\times
L'_m$, $S_m=s_m\times s'_m$, $T_m=t_m\times t'_m$, $I_m=1_m\times
1'_m$, and $\bigcirc_p=\circ_p\times \circ'_p$. The compositions
$\bigcirc_p$ coincide with the unique compositions that complete
the $n$-globular vector space with identities, thus providing a
linear $n$-category. It is straightforwardly checked that the
product of linear $n$-categories verifies the universal property
for binary products.

\begin{prop}\label{3-CatStr} The category {\tt Vect}\,$2$-{\tt Cat} admits a 3-categorical
structure. More precisely, its 2-cells are the linear natural
2-transformations and its 3-cells are the linear 2-modifications.
\end{prop}

This proposition is the linear version (with similar proof) of the
well-known result that the category $2$-{\tt Cat} is a 3-category
with 2-categories as 0-cells, 2-functors as 1-cells, natural
2-transformations as 2-cells, and 2-modifications as 3-cells. The
definitions of $n$-categories and $2$-functors are similar to
those given above in the linear context (but they are formulated
without the use of set theoretical concepts). As for (linear)
natural $2$-transformations and (linear) 2-modifications, let us
recall their definition in the linear setting:

\begin{defi} A {\bf linear natural 2-transformation} $\theta: F\Rightarrow G$ between two
linear 2-functors $F,G:{\cal C}\to{\cal D}$, between the same two
linear 2-categories, assigns to any $a\in{\cal C}_0$ a unique
$\theta_a:F(a)\to G(a)$ in ${\cal D}_1$, linear with respect to
$a$ and such that for any $\za:f\Rightarrow g$ in ${\cal C}_2$,
$f,g:a\to b$ in ${\cal C}_1,$ we have \be\label{2-naturality}
F(\za)\circ_0
1_{\theta_b}=1_{\theta_a}\circ_0G(\za)\;.\ee\end{defi}

If ${\cal C}=L\times L$ is a product linear 2-category, the last
condition reads $$F(\za,\zb)\circ_0
1_{\theta_{t^2\za,t^2\zb}}=1_{\theta_{s^2\za,s^2\zb}}\circ_0G(\za,\zb),$$
for all $(\za,\zb)\in L_2\times L_2$. As functors respect
composition, i.e. as
$$F(\za,\zb)=F(\za\circ_01^2_{t^2\za},1^2_{s^2\zb}\circ_0\zb)=F(\za,1^2_{s^2\zb})\circ_0
F(1^2_{t^2\za},\zb),$$ this naturality condition is verified if
and only if it holds true in case all but one of the 2-cells are
identities $1^2_{-}$, i.e. if and only if the transformation is
natural with respect to all its arguments separately.

\begin{defi} Let $\cal{C,D}$ be two linear
2-categories. A {\bf linear 2-modification} $\zm:\zh\Rrightarrow
\ze$ between two linear natural 2-transformations
$\zh,\ze:F\Rightarrow G$, between the same two linear 2-functors
$F,G:\cal C\to \cal D$, assigns to any object $a\in{\cal C}_0$ a
unique $\zm_a:\zh_a\Rightarrow\ze_a$ in ${\cal D}_2$, which is
linear with respect to $a$ and such that, for any
$\za:f\Rightarrow g$ in ${\cal C}_2$, $f,g:a\to b$ in ${\cal
C}_1$, we have \be \label{2-modification}
F(\za)\circ_0\zm_b=\zm_a\circ_0G(\za).\ee\end{defi}

If ${\cal C}=L\times L$ is a product linear 2-category, it
suffices again that the preceding modification property be
satisfied for tuples $(\za,\zb),$ in which all but one 2-cells are
identities $1^2_{-}$. The explanation is the same as for natural
transformations.\medskip

Beyond linear $2$-functors, linear natural $2$-transformations,
and linear $2$-modifications, we use below multilinear cells.
Bilinear cells e.g., are cells on a product linear 2-category,
with linearity replaced by bilinearity. For instance,

\begin{defi} Let $L$, $L'$, and $L''$ be linear $2$-categories. A
{\bf bilinear 2-functor} $F:L\times L'\to L''$ is a $2$-functor
such that $F:L_m\times L'_m\to L''_m$ is bilinear for all
$m\in\{0,1,2\}$.\end{defi}

Similarly,

\begin{defi} Let $L$, $L'$, and $L''$ be linear $2$-categories. A {\bf bilinear natural 2-transformation} $\theta: F\Rightarrow G$ between two
bilinear 2-functors $F,G:L\times L'\to L''$, assigns to any
$(a,b)\in L_0\times L'_0$ a unique $\theta_{(a,b)}:F(a,b)\to
G(a,b)$ in $L''_1$, which is bilinear with respect to $(a,b)$ and
such that for any $(\za,\zb):(f,h)\Rightarrow (g,k)$ in $L_2\times
L_2'$, $(f,h),(g,k):(a,b)\to (c,d)$ in $L_1\times L'_1,$ we have
\be\label{2-naturality} F(\za,\zb)\circ_0
1_{\theta_{(c,d)}}=1_{\theta_{(a,b)}}\circ_0G(\za,\zb)\;.\ee\end{defi}

\section{Homotopy Lie algebras and categorified Lie algebras}

We now recall the definition of a Lie infinity (strongly homotopy
Lie, sh Lie, $L_{\infty}-$) algebra and specify it in the case of
a 3-term Lie infinity algebra.

\begin{defi} A {\bf Lie infinity algebra} is an $\N$-graded vector space
$V=\oplus_{i\in\N}V_i$ together with a family
$(\ell_i)_{i\in\N^*}$ of graded antisymmetric $i$-linear weight
$i-2$ maps on $V$, which verify the sequence of conditions
\be\label{LieInftyCond} \sum_{i+j=n+1}\sum_{(i,n-i)\mbox{ --
shuffles }
\zs}\chi(\zs)(-1)^{i(j-1)}\ell_j(\ell_i(a_{\zs_1},\ldots,a_{\zs_i}),a_{\zs_{i+1}},\ldots,a_{\zs_n})=0,\ee
where $n\in\{1,2,\ldots\}$, where $\chi(\zs)$ is the product of
the signature of $\zs$ and the Koszul sign defined by $\zs$ and
the homogeneous arguments $a_1,\ldots,a_n\in V$.\end{defi}

For $n=1$, the $L_{\infty}$-condition (\ref{LieInftyCond}) reads
$\ell_1^2=0$ and, for $n=2$, it means that $\ell_1$ is a graded
derivation of $\ell_2$, or, equivalently, that $\ell_2$ is a chain
map from $(V\otimes V,\ell_1\otimes \op{id}+\op{id}\otimes\;
\ell_1)$ to $(V,\ell_1).$\medskip

In particular,

\begin{defi} A {\bf 3-term Lie infinity algebra} is a 3-term graded vector
space $V=V_0\oplus V_1\oplus V_2$ endowed with graded
antisymmetric $p$-linear maps $\ell_p$ of weight $p-2$,
\be\begin{array}{ll} \ell_1:V_{i}\to V_{i-1}&(1\le i\le
2),\\\ell_2:V_i\times V_j\to V_{i+j}&(0\le i+j\le 2),\\
\ell_3:V_i\times V_j\times V_k\to V_{i+j+k+1}&(0\le i+j+k\le
1),\\\ell_4:V_0\times V_0\times V_0\times V_0\to V_2\end{array}\ee
(all structure maps $\ell_p$, $p>4$, necessarily vanish), which
satisfy $L_{\infty}$-condition (\ref{LieInftyCond}) (that is
trivial for all $n>5$).\end{defi}

In this 3-term situation, each $L_{\infty}$-condition splits into
a finite number of equations determined by the various
combinations of argument degrees, see below.\medskip

On the other hand, we have the

\begin{defi}\label{Lie3Alg} A {\bf Lie 3-algebra} is a {\bf linear 2-category} $L$ endowed
with a {\bf bracket}, i.e. an antisymmetric bilinear 2-functor
$[-,-]:L\times L\to L$, which verifies the Jacobi identity up to a
{\bf Jacobiator}, i.e. a skew-symmetric trilinear natural
2-transformation
\be\label{Jacobiator}J_{xyz}:[[x,y],z]\rightarrow[[x,z],y]+[x,[y,z]],\ee
$x,y,z\in L_0$, which in turn satisfies the Baez-Crans Jacobiator
identity up to an {\bf Identiator}, i.e. a skew-symmetric
quadrilinear 2-modification $$\zm_{xyzu}:
[J_{x,y,z},1_u]\circ_0(J_{[x,z],y,u}+J_{x,[y,z],u})\circ_0([J_{xzu},1_y]+1)\circ_0([1_x,J_{yzu}]+1)$$
\be\label{Identiator}\Rightarrow
J_{[x,y],z,u}\circ_0([J_{xyu},1_z]+1)\circ_0(J_{x,[y,u],z}+J_{[x,u],y,z}+J_{x,y,[z,u]}),\ee
$x,y,z,u\in L_0$, required to verify the {\bf coherence law}
\be\label{CohLaw0}\za_1+\za_4^{-1}=\za_3+\za_2^{-1},\ee where
$\za_1$ -- $\za_4$ are explicitly given in Definitions
\ref{Alpha1} -- \ref{Alpha4} and where superscript $-1$ denotes the inverse for composition along a 1-cell.\end{defi}

Just as the {\it Jacobiator} is a natural transformation between
the two sides of the {\it Jacobi} identity, the {\it Identiator}
is a modification between the two sides of the Baez-Crans
Jacobiator {\it identity}.\medskip

In this definition ``skew-symmetric 2-transformation'' (resp.
``skew-symmetric 2-modification'') means that, if we identify
$L_m$ with $\oplus_{i=0}^m V_i$, $V_i=\ker s_i$, as in Proposition
\ref{UniqueCatStr}, the $V_1$-component of $J_{xyz}\in L_1$ (resp.
the $V_2$-component of $\zm_{xyzu}\in L_2$) is antisymmetric.
Moreover, the definition makes sense, as the source and target in
Equation (\ref{Identiator}) are quadrilinear natural
2-transformations between quadrilinear 2-functors from $L^{\times
4}$ to $L.$ These 2-functors are simplest obtained from the
{\small RHS} of Equation (\ref{Identiator}). Further, the
mentioned source and target actually are natural
2-transformations, since a 2-functor composed (on the left or on
the right) with a natural 2-transformation is again a
2-transformation.

\section{Lie 3-algebras in comparison with 3-term Lie infinity algebras}

\begin{rem}In the following, we systematically identify the vector spaces $L_m$, $m\in\{0,\ldots,n\}$, of a
linear $n$-category with the spaces $L_m'=\oplus_{i=0}^mV_i,$
$V_i=\ker s_i$, so that the categorical structure is given by
Equations (\ref{source}) -- (\ref{composition}). In addition, we
often substitute common, index-free notations $($e.g. $\za=(x,{\bf
f},{\bf a})$$)$ for our notations $($e.g. $v=(v_0,v_1,v_2)\in
L_2$$)$. \label{BasicIdentification}\end{rem}

The next theorem is the main result of this paper.

\begin{theo}\label{MainTheo} There exists a 1-to-1 correspondence between Lie 3-algebras and 3-term Lie infinity algebras $(V,\ell_p),$
whose structure maps $\ell_2$ and $\ell_3$ vanish on $V_1\times
V_1$ and on triplets of total degree 1, respectively.
\end{theo}

\begin{ex} There exists a 1-to-1 correspondence between $(n+1)$-term Lie infinity algebras $V=V_0\oplus V_n$ (whose intermediate terms vanish), $n\ge 2$, and $(n+2)$-cocycles of Lie
algebras endowed with a linear representation, see \cite{BC04}, Theorem 6.7. A 3-term Lie infinity algebra implemented by a 4-cocycle can therefore be viewed as a special case of a Lie 3-algebra.\end{ex}

The proof of Theorem \ref{MainTheo} consists of five lemmas.

\subsection{Linear 2-category -- three term chain complex of vector spaces}

First, we recall the correspondence between the underlying structures of a Lie 3-algebra and a 3-term Lie infinity algebra.

\begin{lem} There is a bijective correspondence between linear 2-categories $L$ and 3-term chain complexes of vector spaces $(V,\ell_1)$.\end{lem}

\begin{proof} In the proof of Proposition \ref{EquivCat}, we associated to any linear $2$-category $L$ a unique 3-term chain complex of vector spaces
${\frak N}(L)=V$, whose spaces are given by $V_m=\ker s_m$,
$m\in\{0,1,2\}$, and whose differential $\ell_1$ coincides on
$V_m$ with the restriction $t_m|_{V_m}$. Conversely, we assigned
to any such chain complex $V$ a unique linear 2-category ${\frak
G}(V)=L$, with spaces $L_m=\oplus_{i=0}^mV_i$, $m\in \{0,1,2\}$
and target $t_0(x)=0,$ $t_1(x,{\bf f})=x+\ell_1{\bf f}$,
$t_2(x,{\bf f},{\bf a})=(x,{\bf f}+\ell_1{\bf a})$. In view of
Remark \ref{BasicIdentification}, the maps $\frak N$ and $\frak G$
are inverses of each other.\end{proof}

\begin{rem} The globular space condition is the categorical counterpart of $L_{\infty}$-condition $n=1$.\end{rem}

\subsection{Bracket -- chain map}

We assume that we already built $(V,\ell_1)$ from $L$ or $L$ from $(V,\ell_1)$.

\begin{lem}\label{Lem2} There is a bijective correspondence between antisymmetric bilinear 2-functors $[-,-]$ on $L$ and graded antisymmetric chain maps $\ell_2:(V\otimes V,\ell_1\otimes\op{id}+\op{id}\otimes \ell_1)\to (V,\ell_1)$ that vanish on $V_1\times V_1$.\end{lem}

\begin{proof} Consider first an antisymmetric bilinear ``2-map'' $[-,-]:L\times L\to L$ that verifies all functorial requirements except as concerns composition. This bracket then respects the compositions, i.e., for each pairs $(v,w), (v',w')\in L_m\times L_m$, $m\in\{1,2\}$, that are composable along a $p$-cell, $0\le p<m$, we have \be\label{FunCompGen}[v\circ_pv',w\circ_pw']=[v,w]\circ_p[v',w'],\ee if and only if the following conditions hold true, for any ${\bf f,g}\in V_1$ and any ${\bf a,b}\in V_2$: \be\label{FunComp1}[{\bf f},{\bf g}]=[1_{t{\bf f}},{\bf g}]=[{\bf f},1_{t{\bf g}}],\ee \be\label{FunComp2}[{\bf a},{\bf b}]=[1_{t{\bf a}},{\bf b}]=[{\bf a},1_{t{\bf b}}]=0,\ee \be\label{FunComp3}[1_{{\bf f}},{\bf b}]=[1^2_{t{\bf f}},{\bf b}]=0.\ee To prove the first two conditions, it suffices to compute $[{\bf f}\circ_01_{t{\bf f}},1_0\circ_0{\bf g}]$, for the next three conditions, we consider $[{\bf a}\circ_11_{t{\bf a}},1_0\circ_1{\bf b}]$ and $[{\bf a}\circ_01^2_{0},1^2_0\circ_0{\bf b}]$, and for the last two, we focus on $[1_{\bf f}\circ_01^2_{t{\bf f}},1^2_0\circ_0{\bf b}]$ and $[1_{\bf f}\circ_0(1^2_{t{\bf f}}+1_{\bf f'}),{\bf b}\circ_0{\bf b'}].$ Conversely, it can be straightforwardly checked that Equations (\ref{FunComp1}) -- (\ref{FunComp3}) entail the general requirement (\ref{FunCompGen}).\medskip

On the other hand, a graded antisymmetric bilinear weight 0 map
$\ell_2:V\times V\to V$ commutes with the differentials $\ell_1$
and $\ell_1\otimes\op{id}+\op{id}\otimes \ell_1$, i.e., for all
$v,w\in V$, we have \be\label{ChainGen}
\ell_1(\ell_2(v,w))=\ell_2(\ell_1v,w)+(-1)^v\ell_2(v,\ell_1w)\ee (we assumed that $v$ is homogeneous and denoted its degree by $v$ as well),
if and only if, for any $y\in V_0$, ${\bf f,g}\in V_1$, and ${\bf
a}\in V_2$, \be\label{Chain1}\E_1(\E_2({\bf f},y))=\E_2(\E_1{\bf
f},y),\ee \be\label{Chain2} \E_1(\E_2({\bf f},{\bf
g}))=\E_2(\E_1{\bf f},{\bf g})-\E_2({\bf f},\E_1{\bf g}),\ee
\be\label{Chain3}\E_1(\E_2({\bf a},y))=\E_2(\E_1{\bf a},y),\ee
\be\label{Chain4}0=\E_2(\E_1{\bf f},{\bf b})-\E_2({\bf f},\E_1{\bf
b}).\ee

\begin{rem} Note that, in the correspondence $\ell_1\leftrightarrow t$ and $\ell_2\leftrightarrow [-,-]$, Equations
(\ref{Chain1}) and (\ref{Chain3}) read as compatibility
requirements of the bracket with the target and that Equations
(\ref{Chain2}) and (\ref{Chain4}) correspond to the second
conditions of Equations (\ref{FunComp1}) and (\ref{FunComp3}),
respectively.\end{rem}

{\it Proof of Lemma \ref{Lem2} (continuation)}. To prove the
announced 1-to-1 correspondence, we first define a graded
antisymmetric chain map ${\frak N}([-,-])=\E_2$, $\ell_2:V\otimes
V\to V$ from any antisymmetric bilinear 2-functor $[-,-]:L\times
L\to L$.

Let $x,y\in V_0$, ${\bf f,g}\in V_1$, and ${\bf a,b}\in V_2$. Set
$\E_2(x,y)=[x,y]\in V_0$ and $\E_2(x,{\bf g})=[1_{x},{\bf g}]\in
V_1$. However, we must define $\E_2({\bf f},{\bf g})\in V_2$,
whereas $[{\bf f},{\bf g}]\in V_1.$ Moreover, in this case, the
antisymmetry properties do not match. The observation
$$[{\bf f},{\bf g}]=[1_{t{\bf f}},{\bf g}]=[{\bf f},1_{t{\bf g}}]=\E_2(\E_1{\bf f},{\bf g})=\E_2({\bf f},\ell_1{\bf g})$$ and Condition (\ref{Chain2})
force us to {\it define $\E_2$ on $V_1\times V_1$ as a symmetric
bilinear map valued in $V_2\cap\ker \E_1.$} We further set
$\E_2(x,{\bf b})=[1^2_{x},{\bf b}]\in V_2,$ and, as $\E_2$ is
required to have weight 0, we must set $\E_2({\bf f},{\bf b})=0$
and $\E_2({\bf a},{\bf b})=0.$ It then follows from the functorial
properties of $[-,-]$ that the conditions (\ref{Chain1}) --
(\ref{Chain3}) are verified. In view of Equation (\ref{FunComp3}),
Property (\ref{Chain4}) reads
$$0=[1^2_{t{\bf f}},{\bf b}]-\E_2({\bf f},\E_1{\bf b})=-\E_2({\bf f},\E_1{\bf b}).$$ In other
words, in addition to the preceding requirement, we must {\it
choose $\E_2$ in a way that it vanishes on $V_1\times V_1$ if
evaluated on a 1-coboundary.} These conditions are satisfied if we
choose $\E_2=0$ on $V_1\times V_1$.\medskip

Conversely, from any graded antisymmetric chain map $\E_2$ that
vanishes on $V_1\times V_1$, we can construct an antisymmetric
bilinear 2-functor ${\frak G}(\ell_2)=[-,-]$. Indeed, using
obvious notations, we set
$$[x,y]=\E_2(x,y)\in L_0,\,[1_{x},1_{y}]=1_{[x,y]}\in
L_1,\,[1_{x},{\bf g}]=\E_2(x,{\bf g})\in V_1\subset L_1.$$ Again
$[{\bf f},{\bf g}]\in L_1$ cannot be defined as $\E_2({\bf f},{\bf
g})\in V_2$. Instead, if we wish to build a 2-functor, we must set
$$[{\bf f},{\bf g}]=[1_{t{\bf f}},{\bf g}]=[{\bf f},1_{t{\bf
g}}]=\E_2(\E_1{\bf f},{\bf g})=\E_2({\bf f}, \E_1{\bf g})\in
V_1\subset L_1,$$ which is possible in view of Equation
(\ref{Chain2}), {\it if $\ell_2$ is on $V_1\times V_1$ valued in
2-cocycles} (and in particular if it vanishes on this subspace).
Further, we define
$$[1^2_x,1^2_y]=1^2_{[x,y]}\in L_2,\,[1^2_x,1_{{\bf
g}}]=1_{[1_x,{\bf g}]}\in L_2,\,[1^2_x,{\bf b}]=\E_2(x,{\bf b})\in
V_2\subset L_2,\,[1_{{\bf f}},1_{{\bf g}}]=1_{[{\bf f},{\bf
g}]}\in L_2.$$ Finally, we must set $$[1_{{\bf f}},{\bf
b}]=[1^2_{t{\bf f}},{\bf b}]=\E_2(\E_1{\bf f},{\bf b})=0,$$ which
is possible in view of Equation (\ref{Chain4}), {\it if $\ell_2$
vanishes on $V_1\times V_1$ when evaluated on a 1-coboundary} (and
especially if it vanishes on the whole subspace $V_1\times V_1$),
and
$$[{\bf a},{\bf b}]=[1_{t{\bf a}},{\bf b}]=[{\bf a},1_{t{\bf
b}}]=0,$$ which is possible.

It follows from these definitions that the bracket of $\za=(x,{\bf
f},{\bf a})=1^2_x+1_{{\bf f}}+{\bf a}\in L_2$ and $\zb=(y,{\bf
g},{\bf b})=1^2_y+1_{{\bf g}}+{\bf b}\in L_2$ is given by
{\be\label{BracketFromEll2}[\za,\zb]=(\E_2(x,y),\E_2(x,{\bf
g})+\E_2({\bf f},tg),\E_2(x,{\bf b})+\E_2({\bf a},y))\in
L_2,\ee}where $g=(y,{\bf g})$. The brackets of two elements of
$L_1$ or $L_0$ are obtained as special cases of the latter result.

We thus defined an antisymmetric bilinear map $[-,-]$ that assigns
an $i$-cell to any pair of $i$-cells, $i\in\{0,1,2\}$, and that
respects identities and sources. Moreover, since Equations
(\ref{FunComp1}) -- (\ref{FunComp3}) are satisfied, the map
$[-,-]$ respects compositions provided it respects targets. For
the last of the first three defined brackets, the target condition
is verified due to Equation (\ref{Chain1}). For the fourth
bracket, the target must coincide with $[t{\bf f},t{\bf
g}]=\E_2(\E_1{\bf f},\E_1{\bf g})$ and it actually coincides with
$t[{\bf f},{\bf g}]=\E_1\E_2(\E_1{\bf f},{\bf g})=\E_2(\E_1{\bf
f},\E_1{\bf g})$, again in view of (\ref{Chain1}). As regards the
seventh bracket, the target $t[1^2_x,{\bf b}]=\E_1\E_2(x,{\bf
b})=\E_2(x,\E_1{\bf b}),$ due to (\ref{Chain3}), must coincide
with $[1_x,t{\bf b}]=\E_2(x,\E_1{\bf b})$. The targets of the two
last brackets vanish and $[{\bf f},t{\bf b}]=\E_2({\bf
f},\E_1\E_1{\bf b})=0$ and $[t{\bf a},t{\bf b}]=\E_2(\E_1{\bf
a},\E_1\E_1{\bf b})=0.$\medskip

It is straightforwardly checked that the maps $\frak N$ and $\frak
G$ are inverses.\end{proof}

Note that ${\frak N}$ actually assigns to any antisymmetric
bilinear 2-functor a class of graded antisymmetric chain maps that
coincide outside $V_1\times V_1$ and whose restrictions to
$V_1\times V_1$ are valued in 2-cocycles and vanish when evaluated
on a 1-coboundary. The map $\frak N$, with values in chain maps,
is well-defined thanks to a canonical choice of a representative
of this class. Conversely, the values on $V_1\times V_1$ of the
considered chain map cannot be encrypted into the associated
2-functor, only the mentioned cohomological conditions are of
importance. Without the canonical choice, the map $\frak G$ would
not be injective.

\begin{rem} The categorical counterpart of $L_{\infty}$-condition $n=2$ is the functor condition on compositions.\end{rem}

\begin{rem} A 2-term Lie infinity algebra (resp. a Lie 2-algebra) can be viewed as a 3-term Lie infinity algebra (resp. a Lie 3-algebra). The preceding correspondence then of course reduces to the correspondence of \cite{BC04}.\end{rem}

\subsection{Jacobiator -- third structure map}

We suppose that we already constructed $(V,\ell_1,\ell_2)$ from
$(L,[-,-])$ or $(L,[-,-])$ from $(V,\ell_1,\ell_2)$.

\begin{lem}\label{Lem3} There exists a bijective correspondence between skew-symmetric trilinear natural 2-transformations $J:[[-,-],\bullet]\Rightarrow [[-,\bullet],-]+
[-,[-,\bullet]]$ and graded antisymmetric trilinear weight 1 maps
$\ell_3:V^{\times 3}\to V$ that verify $L_{\infty}$-condition
$n=3$ and vanish in total degree 1.\end{lem}

\begin{proof} A skew-symmetric trilinear natural
2-transformation $J:[[-,-],\bullet]\Rightarrow [[-,\bullet],-]+[-,[-,\bullet]]$
is a map that assigns to any $(x,y,z)\in L_0^{\times 3}$ a unique $J_{xyz}:[[x,y],z]\to [[x,z],y]+[x,[y,z]]$
in $L_1$, such that for any $\za=(z,{\bf f},{\bf a})\in L_2$, we have
$$[[1^2_x,1^2_y],\za]\circ_01_{J_{x,y,\,t^2\za}}=1_{J_{x,y,\,s^2\za}}\circ_0\left([[1^2_x,\za],1^2_y]+[1^2_x,[1^2_y,\za]]\right)$$
(as well as similar equations pertaining to naturality with respect
to the other two variables). A short computation shows that the
last condition decomposes into the following two requirements
on the $V_1$- and the $V_2$-component: \be\label{3a}{\bf
J}_{x,y,\,t{\bf  f}}+[1_{[x,y]},{\bf f}]=[[1_x,{\bf
f}],1_y]+[1_x,[1_y,{\bf f}]],\ee
\be\label{3b}[1^2_{[x,y]},{\bf a}]=[[1^2_x,{\bf a}],1^2_y]+[1^2_x,[1^2_y,{\bf a}]].\ee\smallskip

A graded antisymmetric trilinear weight $1$ map $\ell_3: V^{\times 3}\to V$ verifies $L_{\infty}$-condition $n=3$ if
\be\ell_1(\ell_3(u,v,w))+\ell_2(\ell_2(u,v),w)-(-1)^{vw}\ell_2(\ell_2(u,w),v)+(-1)^{u(v+w)}\ell_2(\ell_2(v,w),u)$$ $$+\ell_3(\ell_1(u),v,w)
-(-1)^{uv}\ell_3(\ell_1(v),u,w)+(-1)^{w(u+v)}\ell_3(\ell_1(w),u,v)=0,\label{LieInfty3}\ee for any homogeneous $u,v,w\in V$.
This condition is trivial for any arguments of total degree $d=u+v+w>2$.
For $d=0$, we write $(u,v,w)=(x,y,z)\in V_0^{\times 3}$, for $d=1$,
we consider $(u,v)=(x,y)\in V_0^{\times 2}$ and $w={\bf f}\in V_1$, for $d=2$,
either $(u,v)=(x,y)\in V_0^{\times 2}$ and $w={\bf a}\in V_2$, or
$u=x\in V_0$ and $(v,w)=({\bf f},{\bf g})\in V_1^{\times 2}$, so that Equation (\ref{LieInfty3}) reads
\be\label{31}\ell_1(\ell_3(x,y,z))+\ell_2(\ell_2(x,y),z)-\ell_2(\ell_2(x,z),y)+\ell_2(\ell_2(y,z),x)=0,\ee
\be\label{32}\ell_1(\ell_3(x,y,{\bf f}))+\ell_2(\ell_2(x,y),{\bf f})-\ell_2(\ell_2(x,{\bf f}),y)+\ell_2(\ell_2(y,{\bf f}),x)
+\ell_3(\ell_1({\bf f}),x,y)=0,\ee
\be\ell_2(\ell_2(x,y),{\bf a})-\ell_2(\ell_2(x,{\bf a}),y)+\ell_2(\ell_2(y,{\bf a}),x)
+\ell_3(\ell_1({\bf a}),x,y)=0,\label{33}\ee
\be\ell_2(\ell_2(x,{\bf f}),{\bf g})+\ell_2(\ell_2(x,{\bf g}),{\bf f})+\ell_2(\ell_2({\bf f},{\bf g}),x) -\ell_3(\ell_1({\bf f}),x,{\bf g})-\ell_3(\ell_1({\bf g}),x,{\bf f})=0.\label{34}\ee\smallskip

It is easy to associate to any such map $\ell_3$ a unique Jacobiator ${\frak G}(\ell_3)=J$: it suffices to set $J_{xyz}:=([[x,y],z],\ell_3(x,y,z))\in L_1$, for
any $x,y,z\in L_0$. Equation (\ref{31}) means that $J_{xyz}$ has the correct
target. Equations (\ref{3a}) and (\ref{3b}) exactly correspond to
Equations (\ref{32}) and (\ref{33}), respectively, {\it if we
assume that in total degree $d=1$, $\ell_3$ is valued
in 2-cocycles and vanishes when evaluated on a 1-coboundary}. These conditions are verified if we start from a structure map $\ell_3$ that vanishes on any arguments of total degree 1.

\begin{rem} Remark that the values $\ell_3(x,y,{\bf f})\in V_2$ cannot be encoded
in a natural 2-transformation $J:L_0^{\times 3}\ni (x,y,z)\to
J_{xyz}\in L_1$ (and that the same holds true for Equation (\ref{34}), whose first three terms are zero, since we started from a map $\ell_2$ that vanishes on $V_1\times
V_1$).\end{rem}

{\it Proof of Lemma \ref{Lem3} (continuation)}. Conversely, to any Jacobiator $J$ corresponds a unique map ${\frak N}(J)=\ell_3$. Just set $\ell_3(x,y,z):={\bf J}_{xyz}\in V_1$ and $\ell_3(x,y,{\bf f})=0$, for all $x,y,z\in V_0$ and ${\bf f}\in V_1$ (as $\ell_3$ is required to have weight 1, it must vanish if evaluated on elements of degree $d\ge 2$).\smallskip

Obviously the composites $\frak{NG}$ and $\frak{GN}$ are identity maps.\end{proof}

\begin{rem} The naturality condition is, roughly speaking, the categorical analogue of the $L_{\infty}$-condition $n=3$.\end{rem}

\subsection{Identiator -- fourth structure map}

For $x,y,z,u\in L_0$, we set
\be\label{Eta}\zh_{xyzu}:=[J_{x,y,z},1_u]\circ_0(J_{[x,z],y,u}+J_{x,[y,z],u})\circ_0([J_{xzu},1_y]+1)\circ_0([1_x,J_{yzu}]+1)\in L_1\ee and
\be\label{Epsilon}\ze_{xyzu}:=
J_{[x,y],z,u}\circ_0([J_{xyu},1_z]+1)\circ_0(J_{x,[y,u],z}+J_{[x,u],y,z}+J_{x,y,[z,u]})\in L_1,\ee see Definition \ref{Lie3Alg}. The identities $1$ are uniquely determined by the sources of the involved factors. The quadrilinear natural 2-transformations $\zh$ and $\ze$ are actually the left and right hand composites of the Baez-Crans octagon that pictures the coherence law of a Lie 2-algebra, see \cite{BC04}, Definition 4.1.3. They connect the quadrilinear 2-functors $F,G:L\times L\times L\times L\to L$, whose values at $(x,y,z,u)$ are given by the source and the target of the 1-cells $\zh_{xyzu}$ and $\ze_{xyzu}$, as well as by the top and bottom sums of triple brackets of the mentioned octagon.

\begin{lem} The skew-symmetric quadrilinear 2-modifications $\zm:\zh\Rrightarrow \ze$ are in 1-to-1 correspondence with the graded antisymmetric quadrilinear weight 2 maps $\ell_4:V^{\times 4}\to V$ that verify the $L_{\infty}$-condition $n=4$.\end{lem}

\begin{proof}  A skew-symmetric quadrilinear 2-modification
$\zm:\zh\Rrightarrow \ze$ maps every tuple $(x,y,z,u)\in
L_0^{\times 4}$ to a unique $\zm_{xyzu}:\zh_{xyzu}\Rightarrow
\ze_{xyzu}$ in $L_2$, such that, for any
$\za=(u,{\bf f},{\bf a})\in L_2$, we have
\be F(1^2_x,1^2_y,1^2_z,\za)\circ_0\zm_{x,y,z,u+t{\bf f}}=\zm_{xyzu}\circ_0G(1^2_x,1^2_y,1^2_z,\za)\label{ModCond}\ee (as well as similar
results concerning naturality with respect to the three other
variables). If we decompose $\zm_{xyzu}\in L_2=V_0\oplus V_1\oplus V_2$,
$$\zm_{xyzu}=(F(x,y,z,u),{\bf h}_{xyzu},{\bf m}_{xyzu})=1_{\zh_{xyzu}}+{\bf m}_{xyzu},$$ Condition (\ref{ModCond}) reads
\be\label{41}F(1_x,1_y,1_z,{\bf f})+{\bf h}_{x,y,z,u+t{\bf f}}={\bf h}_{xyzu}+G(1_x,1_y,1_z,{\bf f}),\ee
\be\label{42}F(1^2_x,1^2_y,1^2_z,{\bf a})+{\bf m}_{x,y,z,u+t{\bf f}}={\bf m}_{xyzu}+G(1^2_x,1^2_y,1^2_z,{\bf a}).\ee\smallskip

On the other hand, a graded antisymmetric quadrilinear weight 2 map $\ell_4:V^{\times 4}\to V$, and more
precisely $\ell_4:V_0^{\times 4}\to V_2$, verifies $L_{\infty}$-condition $n=4$, if
$$\ell_1(\ell_4(a,b,c,d))$$
$$-\ell_2(\ell_3(a,b,c),d)+(-1)^{cd}\ell_2(\ell_3(a,b,d),c)-(-1)^{b(c+d)}\ell_2(\ell_3(a,c,d),b)$$
$$+(-1)^{a(b+c+d)}\ell_2(\ell_3(b,c,d),a)+
\ell_3(\ell_2(a,b),c,d)-(-1)^{bc}\ell_3(\ell_2(a,c),b,d)$$
$$+(-1)^{d(b+c)}\ell_3(\ell_2(a,d),b,c)+(-1)^{a(b+c)}\ell_3(\ell_2(b,c),a,d)$$
$$-(-1)^{ab+ad+cd}\ell_3(\ell_2(b,d),a,c)+(-1)^{(a+b)(c+d)}\ell_3(\ell_2(c,d),a,b)$$
$$-\ell_4(\ell_1(a),b,c,d)+(-1)^{ab}\ell_4(\ell_1(b),a,c,d)$$
\be\label{LieInfty4}-(-1)^{c(a+b)}\ell_4(\ell_1(c),a,b,d)+(-1)^{d(a+b+c)}\ell_4(\ell_1(d),a,b,c)=0,\ee for all homogeneous $a,b,c,d\in V$.
The condition is trivial for $d\ge 2$. For $d=0$, we write
$(a,b,c,d)=(x,y,z,u)\in V_0^{\times 4}$, and, for $d=1$, we take
$(a,b,c,d)=(x,y,z,{\bf f})\in V_0^{\times 3}\times V_1$, so that -- since $\ell_2$ and $\ell_3$ vanish on $V_1\times
V_1$ and for $d=1$, respectively -- Condition (\ref{LieInfty4}) reads
\be\label{4a}\ell_1(\ell_4(x,y,z,u))-{\bf h}_{xyzu}+{\bf e}_{xyzu}=0,\ee
\be\label{4b} \ell_4(\ell_1({\bf f}),x,y,z)=0,\ee where
${\bf h}_{xyzu}$ and ${\bf e}_{xyzu}$ are the $V_1$-components of $\zh_{xyzu}$ and $\ze_{xyzu}$, see Equations (\ref{Eta}) and (\ref{Epsilon}).\medskip

We can associate to any such map $\ell_4$ a unique 2-modification ${\frak G}(\ell_4)=\zm$, $\zm:\zh\Rrightarrow \ze$. It suffices to set, for $x,y,z,u\in
L_0$, $$\zm_{xyzu}=(F(x,y,z,u),{\bf h}_{xyzu},-\ell_4(x,y,z,u))\in
L_2.$$ In view of Equation (\ref{4a}), the target of this 2-cell is
$$t\zm_{xyzu}=(F(x,y,z,u),{\bf h}_{xyzu}-\ell_1(\ell_4(x,y,z,u)))=\ze_{xyzu}\in L_1.$$ Note now that the
2-naturality equations (\ref{3a}) and (\ref{3b}) show that 2-naturality of $\zh:F\Rightarrow G$ means that
$$F(1_x,1_y,1_z,{\bf f})+{\bf h}_{x,y,z,u+t{\bf f}}={\bf h}_{xyzu}+G(1_x,1_y,1_z,{\bf f}),$$
$$F(1^2_x,1^2_y,1^2_z,{\bf a})=G(1^2_x,1^2_y,1^2_z,{\bf a}).$$
When comparing with Equations (\ref{41}) and (\ref{42}), we conclude
that $\zm$ is a 2-modification if and only if $\ell_4(\ell_1({\bf f}),x,y,z)=0,$ which is exactly Equation (\ref{4b}).\medskip

Conversely, if we are given a skew-symmetric quadrilinear 2-modification
$\zm:\zh\Rrightarrow\ze$, we
define a map ${\frak N}(\zm)=\ell_4$ by setting
$\ell_4(x,y,z,u)=-{\bf m}_{xyzu}$, with self-explaining notations. $L_{\infty}$-condition $n=4$ is equivalent with Equations
(\ref{4a}) and (\ref{4b}). The first means that $\zm_{xyzu}$ must
have the target $\ze_{xyzu}$ and the second requires that ${\bf m}_{t{\bf f},x,y,z}$ vanish -- a consequence of the 2-naturality of
$\zh$ and of Equation (\ref{42}).\medskip

The maps $\frak{N}$ and $\frak{G}$ are again inverses.\end{proof}

\subsection{Coherence law -- $L_{\infty}$-condition $n=5$}

\begin{lem} Coherence law (\ref{CohLaw0}) is equivalent to $L_{\infty}$-condition $n=5$.\end{lem}

\begin{proof} The sh Lie condition $n=5$ reads,
$$\ell_2(\ell_4(x,y,z,u),v)-\ell_2(\ell_4(x,y,z,v),u)+\ell_2(\ell_4(x,y,u,v),z)-\ell_2(\ell_4(x,z,u,v),y)+\ell_2(\ell_4(y,z,u,v),x)$$
$$+\ell_4(\ell_2(x,y),z,u,v)-\ell_4(\ell_2(x,z),y,u,v)+\ell_4(\ell_2(x,u),y,z,v)-\ell_4(\ell_2(x,v)y,z,u)+\ell_4(\ell_2(y,z),x,u,v)$$
$$-\ell_4(\ell_2(y,u),x,z,v)+\ell_4(\ell_2(y,v),x,z,u)+\ell_4(\ell_2(z,u),x,y,v)-\ell_4(\ell_2(z,v),x,y,u)+\ell_4(\ell_2(u,v),x,y,z)$$
\be=0,\label{51}\ee for any $x,y,z,u,v\in V_0$. It is trivial in degree $d\ge 1$. Let us mention that it follows from Equation (\ref{31}) that $(V_0,\ell_2)$ is
a Lie algebra up to homotopy, and from Equation (\ref{33}) that $\ell_2$ is a
representation of $V_0$ on $V_2$. Condition (\ref{51}) then requires
that $\ell_4$ be a Lie algebra 4-cocycle of $V_0$ represented upon
$V_2$.\medskip

The coherence law for the 2-modification $\zm$ corresponds to four different ways
to rebracket the expression $F([x,y],z,u,v)=[[[[x,y],z],u],v]$ by
means of $\zm$, $J$, and $[-,-]$. More precisely, we define, for any tuple
$(x,y,z,u,v)\in L_0^{\times 5}$, four 2-cells
$$\za_i:\zs_i\Rightarrow \zt_i,$$ $i\in\{1,2,3,4\}$, in $L_2$, where
$\zs_i,\zt_i:A_i\to B_i$. Dependence on the considered tuple is
understood. We omit temporarily also index $i$. Of course, $\zs$ and $\zt$ read $\zs=(A,{\bf s})\in L_1$ and $\zt=(A,{\bf t})\in L_1$. $$\mbox{If }\,\za=(A,{\bf s},{\bf a})\in L_2,\,\mbox{ we set }\,\za^{-1}=(A,{\bf t},-{\bf a})\in L_2,$$ which is,
as easily seen, the inverse of $\za$ for composition along
1-cells.

\begin{defi} The {\bf coherence law} for the 2-modification $\zm$ of a Lie 3-algebra $(L,[-,-],J,\zm)$ reads
\be\label{CohLaw}\za_1+\za_4^{-1}=\za_3+\za_2^{-1},\ee where $\za_1$ -- $\za_4$ are detailed in the next definitions.\end{defi}

\begin{defi}\label{Alpha1} The {\bf first 2-cell} $\za_1$ is given by \be\label{CohLaw11} \za_1=
1_{11}\circ_0\left(\zm_{x,y,z,[u,v]}+[\zm_{xyzv},1^2_u]\right)\circ_01_{12}\circ_0\left(\zm_{[x,v],y,z,u}+\zm_{x,[y,v],z,u}+\zm_{x,y,[z,v],u}+1^2\right),\ee
where \be\label{CohLaw12}
1_{11}=1_{J_{[[x,y],z],u,v}},1_{12}=1_{[J_{x,[z,v],y},1_{u}]+[J_{[x,v],z,y},1_{u}]+[J_{x,z,[y,v]},1_{u}]}+1^2,\ee
and where the $1^2$ are the identity 2-cells associated with
the elements of $L_0$ provided by the composability
condition.\end{defi}

For instance, the squared target of the second factor of $\za_1$
is $G(x,y,z,[u,v])+[G(x,y,z,v),u]$, whereas the squared source of
the third factor is
$$[[[x,[z,v]],y],{u}]+[[[[x,v],z],y],{u}]+[[[x,z],[y,v]],{u}]+\ldots.$$
As the three first terms of this sum are three of the six terms of
$[G(x,y,z,v),u]$, the object ``$\ldots$'', at which $1^2$ in $1_{12}$ is evaluated,
is the sum of the remaining terms and $G(x,y,z,[u,v]).$

\begin{defi}\label{Alpha2} The {\bf fourth 2-cell} $\za_4$ is equal to
\be\label{CohLaw21}\za_4=[\zm_{xyzu},1^2_v]\circ_01_{41}\circ_0\left(\zm_{[x,u],y,z,v}+\zm_{x,[y,u],z,v}+\zm_{x,y,[z,u],v}\right)\circ_01_{42},\ee
where $$
1_{41}=1_{[J_{[x,u],z,y},1_v]+[J_{x,z,[y,u]},1_v]+[J_{x,[z,u],y},1_v]}+1^2,$$
\be\label{CohLaw22}
1_{42}=1_{[[J_{xuv},1_z],1_y]+[J_{xuv},1_{[y,z]}]+[1_x,[J_{yuv},1_z]]+[[1_x,J_{zuv}],1_y]+[1_x,[1_y,J_{zuv}]]+[1_{[x,z]},J_{yuv}]}+1^2.\ee\end{defi}

\begin{defi}\label{Alpha3} The {\bf third 2-cell} $\za_3$ reads \be\label{CohLaw31}
\za_3=\zm_{[x,y],z,u,v}\circ_01_{31}\circ_0\left([\zm_{xyuv},1^2_z]+1^2\right)\circ_01_{32}\circ_01_{33},\ee
where $$ 1_{31}=1_{[J_{[x,y],v,u},1_z]}+1^2,$$
$$1_{32}=1_{[J_{xyv},1_{[z,u]}]+J_{x,y,[[z,v],u]}+J_{x,y,[z,[u,v]]}+J_{[[x,v],u],y,z}+J_{[x,v],[y,u],z}+J_{[x,u],[y,v],z}+
J_{x,[[y,v],u],z}+J_{[x,[u,v]],y,z}+J_{x,[y,[u,v]],z}+[J_{xyu},1_{[z,v]}]},$$
\be\label{CohLaw32}
1_{33}=1_{J_{x,[y,v],[z,u]}+J_{[x,v],y,[z,u]}+J_{x,[y,u],[z,v]}+J_{[x,u],y,[z,v]}}+1^2.\ee\end{defi}

\begin{defi}\label{Alpha4} The {\bf second 2-cell} $\za_2$ is defined as \be\label{CohLaw41}
\za_2=1_{21}\circ_0\left(\zm_{[x,z],y,u,v}+\zm_{x,[y,z],u,v}\right)\circ_01_{22}\circ_0\left([1^2_x,\zm_{yzuv}]+[\zm_{xzuv},1^2_y]+1^2\right)\circ_01_{23},\ee
where \be\label{CohLaw42} 1_{21}=1_{[[J_{xyz},1_u],1_v]},
1_{22}=1_{[1_x,J_{[y,z],v,u}]+[J_{[x,z],v,u},1_y]}+1^2,
1_{23}=1_{[J_{xzv},1_{[y,u]}]+[J_{xzu},1_{[y,v]}]+[1_{[x,v]},J_{yzu}]+[1_{[x,u]},J_{yzv}]}+1^2.\ee\end{defi}

To get the component expression
\be\label{CohLawComp}(A_1+A_4,{\bf s}_1+{\bf t}_4,{\bf a}_1-{\bf a}_4)=(A_3+A_2,{\bf s}_3+{\bf t}_2,{\bf a}_3-{\bf a}_2)\ee of the coherence law (\ref{CohLaw}), we now comment on the computation of the components
$(A_i,{\bf s}_i,{\bf a}_i)$ (resp. $(A_i,{\bf t}_i,-{\bf a}_i)$) of $\za_i$ (resp. $\za_i^{-1}$).\medskip

As concerns $\za_1$, it is straightforwardly seen that all compositions make sense, that its
$V_0$-component is $$A_1=F([x,y],z,u,v),$$ and that the
$V_2$-component is
$${\bf a}_1=$$
$$-\ell_4(x,y,z,\ell_2(u,v))-\ell_2(\ell_4(x,y,z,v),u)-\ell_4(\ell_2(x,v),y,z,u)-\ell_4(x,\ell_2(y,v),z,u)-\ell_4(x,y,\ell_2(z,v),u).$$
When actually examining the composability conditions, we find that
$1^2$ in the fourth factor of $\za_1$ is $1^2_{G(x,y,z,[u,v])}$
and thus that the target $t^2\za_1$ is made up by the 24 terms
$$G([x,v],y,z,u)+G(x,[y,v],z,u)+G(x,y,[z,v],u)+G(x,y,z,[u,v]).$$
The computation of the $V_1$-component ${\bf s}_1$ is tedious but
simple -- it leads to a sum of 29 terms of the type
``$\ell_3\ell_2\ell_2,$ $\ell_2\ell_3\ell_2$, or
$\ell_2\ell_2\ell_3$''. We will comment on it in the case of
$\za_4^{-1}$, which is slightly more interesting.\medskip

The $V_0$-component of $\za_4^{-1}$ is
$$A_4=[F_{xyzu},v]=F([x,y],z,u,v)$$ and its $V_2$-component is
equal to
$$-{\bf a}_4=\ell_2(\ell_4(x,y,z,u),v)+\ell_4(\ell_2(x,u),y,z,v)+\ell_4(x,\ell_2(y,u),z,v)+\ell_4(x,y,\ell_2(z,u),v).$$
The $V_1$-component ${\bf t}_4$ of $\za_4^{-1}$ is the
$V_1$-component of the target of $\za_4$. This target is the
composition of the targets of the four factors of $\za_4$ and its
$V_1$-component is given by
$${\bf t}_4=[{\bf e}_{xyzu},1_v]+[{\bf J}_{[x,u],z,y},1_v]+[{\bf J}_{x,z,[y,u]},1_v]+[{\bf J}_{x,[z,u],y},1_v]+{\bf e}_{[x,u],y,z,v}+{\bf e}_{x,[y,u],z,v}+{\bf e}_{x,y,[z,u],v}$$
$$+[[{\bf J}_{xuv},1_z],1_y]+[{\bf J}_{xuv},1_{[y,z]}]+[1_x,[{\bf J}_{yuv},1_z]]+[[1_x,{\bf J}_{zuv}],1_y]+[1_x,[1_y,{\bf J}_{zuv}]]+[1_{[x,z]},{\bf J}_{yuv}].$$ The definition (\ref{Epsilon}) of $\ze$ immediately provides its
$V_1$-component $\bf e$ as a sum of 5 terms of the type ``$\ell_3\ell_2$ or
$\ell_2\ell_3$''. The preceding $V_1$-component ${\bf t}_4$ of
$\za_4^{-1}$ can thus be explicitly written as a sum of $29$ terms
of the type ``$\ell_3\ell_2\ell_2,$ $\ell_2\ell_3\ell_2$, or
$\ell_2\ell_2\ell_3$''. It can moreover be checked that the target
$t^2\za_4^{-1}$ is again a sum of 24 terms -- the same as for $t^2\za_1$.\medskip

The $V_0$-component of $\za_3$ is $$A_3=F([x,y],z,u,v),$$ the $V_1$-component
${\bf s}_3$ can be computed as before and is a sum of 25 terms of
the usual type ``$\ell_3\ell_2\ell_2,$ $\ell_2\ell_3\ell_2$, or
$\ell_2\ell_2\ell_3$'', whereas the $V_2$-component is equal to
$${\bf a}_3=-\ell_4(\ell_2(x,y),z,u,v)-\ell_2(\ell_4(x,y,u,v),z).$$
Again $t^2\za_3$ is made up by the same 24 terms as
$t^2\za_1$ and $t^2\za_4^{-1}$.\medskip

Eventually, the $V_0$-component of $\za_2^{-1}$ is $$A_2=F([x,y],z,u,v),$$ the
$V_1$-component ${\bf t}_2$ is straightforwardly obtained as a sum
of 27 terms of the form ``$\ell_3\ell_2\ell_2,$
$\ell_2\ell_3\ell_2$, or $\ell_2\ell_2\ell_3$'', and the
$V_2$-component reads
$$-{\bf a}_2=\ell_4(\ell_2(x,z),y,u,v)+\ell_4(x,\ell_2(y,z),u,v)+\ell_2(x,\ell_4(y,z,u,v))+\ell_2(\ell_4(x,z,u,v),y).$$
The target $t^2\za_2^{-1}$ is the same as in the
preceding cases.\medskip

Coherence condition (\ref{CohLaw}) and its component expression (\ref{CohLawComp}) can now be understood. The condition on the $V_0$-components is obviously trivial. The condition on the $V_2$-components is nothing but
$L_{\infty}$-condition $n=5$, see Equation (\ref{51}). The verification of triviality of the condition on the $V_1$-components is lengthy: 6 pairs (resp. 3 pairs) of terms of the {\small LHS} ${\bf s}_1+{\bf t}_4$ (resp. {\small RHS} ${\bf s}_3+{\bf t}_2$) are opposite and cancel out, 25 terms of
the {\small LHS} coincide with terms of the {\small RHS}, and, finally, 7 triplets of {\small LHS}-terms combine with triplets of {\small RHS}-terms and provide 7 sums of 6 terms, e.g.
$$\ell_3(\ell_2(\ell_2(x,y),z),u,v)+\ell_2(\ell_3(x,y,z),\ell_2(u,v))+\ell_2(\ell_2(\ell_3(x,y,z),v),u)$$ $$-\ell_2(\ell_2(\ell_3(x,y,z),u),v)-\ell_3(\ell_2(\ell_2(x,z),y),u,v)-\ell_3({\ell_2(x,\ell_2(y,z)),u,v}).$$
Since, for ${\bf f}=\ell_3(x,y,z)\in V_1$, we have $$\ell_1({\bf f})=t{\bf J}_{xyz}=\ell_2(\ell_2(x,z),y)+\ell_2(x,\ell_2(y,z))-\ell_2(\ell_2(x,y),z),$$ the preceding sum vanishes in view of Equation
(\ref{32}). Indeed, if we associate a Lie 3-algebra to a 3-term Lie infinity algebra, we started from a homotopy algebra whose term $\ell_3$ vanishes in total degree 1, and if we build an sh algebra from a categorified algebra,
we already constructed an $\ell_3$-map with that property. Finally, the condition on $V_1$-components is really trivial and the coherence law (\ref{CohLaw}) is actually equivalent to $L_{\infty}$-condition $n=5$.\end{proof}

\section{Monoidal structure of the category {\tt Vect} $n$-{\tt
Cat}}

In this section we exhibit a specific aspect of the natural
monoidal structure of the category of linear $n$-categories.

\begin{prop}\label{LinnFun} If $L$ and $L'$ are linear $n$-categories, a family $F_m:L_m\to L_m'$ of linear maps that respects sources, targets, and identities, commutes
automatically with compositions and thus defines a linear
$n$-functor $F:L\to L'$.\end{prop}

\begin{proof} If $v=(v_0,\ldots, v_m), w=(w_0,\ldots,w_m)\in L_m$ are composable along a $p$-cell, then $F_mv=(F_0v_0, \ldots,$ $ F_mv_m)$ and $F_mw=(F_0w_0,
\ldots, F_mw_m)$ are composable as well, and
$F_m(v\circ_pw)=(F_mv)\circ_p(F_mw)$ in view of Equation
(\ref{composition}).\end{proof}

\begin{prop} The category {\tt Vect} $n$-{\tt Cat} admits a canonical symmetric monoidal structure $\boxtimes$.\end{prop}

\begin{proof} We first define the product $\boxtimes$ of two linear $n$-categories $L$ and $L'$. The
$n$-globular vector space that underlies the linear $n$-category
$L\boxtimes L'$ is defined in the obvious way, $(L\boxtimes
L')_m=L_m\otimes L_m'$, $S_m=s_m\otimes s'_m$, $T_m=t_m\otimes
t'_m.$ Identities are clear as well, $I_m=1_m\otimes 1'_m.$ These
data can be completed by the unique possible compositions
$\square_p$ that then provide a linear $n$-categorical structure.

If $F:L\to M$ and $F':L'\to M'$ are two linear $n$-functors, we
set $$(F\boxtimes F')_m=F_m\otimes F'_m\in\op{Hom}_{\K}(L_m\otimes
L'_m,M_m\otimes M'_m),$$ where $\K$ denotes the ground field. Due
to Proposition {\ref{LinnFun}}, the family $(F\boxtimes F')_m$
defines a linear $n$-functor $F\boxtimes F':L\boxtimes L'\to
M\boxtimes M'$.

It is immediately checked that $\boxtimes$ respects composition
and is therefore a functor from the product category ({\tt Vect}
$n$-{\tt Cat})$^{\times 2}$ to {\tt Vect} $n$-{\tt Cat}. Further,
the linear $n$-category $K$, defined by $K_m=\K$,
$s_m=t_m=\op{id}_{\K}$ ($m>0$), and $1_m=\op{id}_{\K}$ ($m<n$),
acts as identity object for $\boxtimes$. Its is now clear that
$\boxtimes$ endows {\tt Vect} $n$-{\tt Cat} with a symmetric
monoidal structure.\end{proof}

\begin{prop}\label{PreUP} Let $L$, $L'$, and $L''$ be linear $n$-categories. For any bilinear
$n$-functor $F:L\times L'\to L''$, there exists a unique linear
$n$-functor $\tilde F:L\boxtimes L'\to L''$, such that
$\boxtimes\, \tilde F=F.$ Here $\boxtimes:L\times L'\to L\boxtimes
L'$ denotes the family of bilinear maps $\boxtimes_m:L_m\times
L_m'\ni (v,v')\mapsto v\otimes v'\in L_m\otimes L_m'$, and
juxtaposition denotes the obvious composition of the first with
the second factor.\end{prop}

\begin{proof} The result is a straightforward consequence of the universal property of the tensor product of vector spaces.\end{proof}

The next remark is essential.

\begin{rem} Proposition \ref{PreUP} is not a Universal Property for the tensor product $\boxtimes$ of {\tt Vect} $n$-{\tt Cat},
since $\boxtimes:L\times L'\to L\boxtimes L'$ is not a bilinear
$n$-functor. It follows that bilinear $n$-functors on a product
category $L\times L'$ cannot be identified with linear
$n$-functors on the corresponding tensor product category
$L\boxtimes L'$.\end{rem}

The point is that the family $\boxtimes_m$ of bilinear maps
respects sources, targets, and identities, but not compositions
(in contrast with a similar family of linear maps, see Proposition
\ref{LinnFun}). Indeed, if $(v,v'),(w,w')\in L_m\times L'_m$ are
two $p$-composable pairs (note that this condition is equivalent
with the requirement that $v,w\in L_m$ and $v',w'\in L_m'$ be
$p$-composable), we have
\be\label{CartProd}\boxtimes_m((v,v')\circ_p(w,w'))=(v\circ_pw)\otimes(v'\circ_pw')\in
L_m\otimes L_m',\ee and
\be\label{TensProd}\boxtimes_m(v,v')\;\circ_p\;\boxtimes_m(w,w')=(v\otimes
v')\circ_p(w\otimes w')\in L_m\otimes L_m'.\ee As the elements
(\ref{CartProd}) and (\ref{TensProd}) arise from the compositions
in $L_m\times L'_m$ and $L_m\otimes L'_m$, respectively, -- which
are forced by linearity and thus involve the completely different
linear structures of these spaces -- it can be expected that the
two elements do not coincide.\medskip

Indeed, when confining ourselves, to simplify, to the case $n=1$
of linear categories, we easily check that
\be\label{NonFunUP}(v\circ w)\otimes (v'\circ w')=(v\otimes v')\circ(w\otimes
w')+(v-1_{tv})\otimes {\bf w'}+{\bf w}\otimes (v'-1_{{tv'}}).\ee

Observe also that the source spaces of the linear maps
$$\circ_L\otimes\circ'_{L'}:(L_1\times_{L_0}L_1)\otimes
(L'_1\times_{L'_0}L'_1)\ni (v,w)\otimes(v',w')\mapsto (v\circ
w)\otimes (v'\circ w')\in L_1\otimes L_1'$$ and
$$\circ_{L\boxtimes L'}:(L_1\otimes L_1')\times_{L_0\otimes L_0'}(L_1\otimes L_1')\ni ((v\otimes v'),(w\otimes w'))\mapsto (v\otimes v')\circ(w\otimes w')\in
L_1\otimes L_1'$$ are connected by
\be \ell_2: (L_1\times_{L_0}L_1)\otimes (L'_1\times_{L'_0}L'_1)\ni
(v,w)\otimes(v',w')\mapsto (v\otimes v',w\otimes w')\in
(L_1\otimes L_1')\times_{L_0\otimes L_0'}(L_1\otimes L_1')\label{NonFun}\ee -- a
linear map with nontrivial kernel.

\section{Discussion}

We continue working in the case $n=1$ and
investigate a more conceptual approach to the construction of a
chain map $\ell_2:{\frak N}(L)\otimes {\frak N}(L)\to {\frak
N}(L)$ from a bilinear functor $[-,-]:L\times L\to L$.\medskip

When denoting by $[-,-]:L\boxtimes L\to L$ the induced linear functor, we get a chain map ${\frak
N}([-,-]):{\frak N}(L\boxtimes L)\to {\frak N}(L)$, so that it is
natural to look for a second chain map $$\zf:{\frak N}(L)\otimes
{\frak N}(L)\to {\frak N}(L\boxtimes L).$$

The informed reader may skip the following subsection.

\subsection{Nerve and normalization functors, Eilenberg-Zilber chain map}

The objects of the simplicial category $\zD$ are the finite
ordinals $n=\{0,\ldots, n-1\}$, $n\ge 0$. Its morphisms $f:m\to n$ are the
order respecting functions between the sets $m$ and $n$. Let
$\zd_i:n\rightarrowtail n+1$ be the injection that omits image
$i$, $i\in\{0,\ldots,n\}$, and let $\zs_i:n+1\twoheadrightarrow n$
be the surjection that assigns the same image to $i$ and $i+1$,
$i\in\{0,\ldots,n-1\}$. Any order respecting function $f:m\to n$
reads uniquely as
$f=\zs_{j_1}\ldots\zs_{j_h}\zd_{i_1}\ldots\zd_{i_k}$, where the
$j_r$ are decreasing and the $i_s$ increasing. The application of this epi-monic decomposition to binary composites $\zd_i\zd_j$, $\zs_i\zs_j$, and $\zd_i\zs_j$ yields
three basic commutation relations.\medskip

A simplicial object in the category {\tt Vect} is a functor
$S\in[\zD^{+\op{op}},\mbox{\tt Vect}]$, where $\zD^{+}$ denotes
the full subcategory of $\zD$ made up by the nonzero finite
ordinals. We write this functor $n+1\mapsto S(n+1)=:{S_n}$, $n\ge 0$,
($S_n$ is the vector space of $n$-simplices), $\zd_i\mapsto
S(\zd_i)=:{d_i:S_n\to S_{n-1}}$, $i\in\{0,\ldots,n\}$ ($d_i$ is a
face operator), $\zs_i\mapsto S(\zs_i)=:{s_i:S_{n}\to S_{n+1}}$,
$i\in\{0,\ldots,n\}$ ($s_i$ is a degeneracy operator). The $d_i$
and $s_j$ verify the duals of the mentioned commutation rules.
The simplicial data ($S_n,d^n_i,s^n_i$) (we added
superscript $n$) of course completely determine the functor $S$.
Simplicial objects in {\tt Vect} form themselves a category,
namely the functor category $s(\mbox{\tt
Vect}):=[\zD^{+\op{op}},\mbox{\tt Vect}]$, for which the
morphisms, called simplicial morphisms, are the natural
transformations between such functors. In view of the epi-monic
factorization, a simplicial map $\za:S\to T$ is exactly a
family of linear maps $\za_n:S_n\to T_n$ that commute with the
face and degeneracy operators. \medskip

The nerve functor $${\cal N}:\mbox{\tt VectCat}\to s(\mbox{\tt
Vect})$$ is defined on a linear category $L$ as the sequence
$L_0,L_1,L_2:=L_1\times_{L_0}L_1,
L_3:=L_1\times_{L_0}L_1\times_{L_0}L_1\ldots$ of vector spaces of
$0,1,2,3\ldots$ simplices, together with the face operators
``composition'' and the degeneracy operators ``insertion of
identity'', which verify the simplicial commutation rules.
Moreover, any linear functor $F:L\to L'$ defines linear maps
$F_n:L_n\ni (v_1,\ldots,v_n)\to (F(v_1),\ldots, F(v_n))\in L'_n$
that implement a simplicial map.\medskip

The normalized or Moore chain complex of a simplicial vector space
$S=(S_n, d_i^n, s_i^n)$ is given by
$N(S)_n=\cap_{i=1}^n\op{ker}d_i^n\subset S_n$ and
$\partial_n=d_0^n.$ Normalization actually provides a functor
$$ N:s(\mbox{\tt Vect}) \leftrightarrow \mbox{\tt C}^+(\mbox{\tt
Vect}):\zG$$ valued in the category of nonnegatively graded chain
complexes of vector spaces. Indeed, if $\za:S\to T$ is a
simplicial map, then $\za_{n-1}d_{i}^n=d_{i}^n\za_{n}$. Thus,
$N(\za):N(S)\to N(T$), defined on $c_n\in N(S)_n$ by
$N(\za)_n(c_n)=\za_n(c_n)$, is valued in $N(T)_n$ and is further a
chain map. Moreover, the Dold-Kan correspondence claims that the
normalization functor $N$ admits a right adjoint $\zG$ and that
these functors combine into an equivalence of categories.\medskip

It is straightforwardly seen that, for any linear category $L$, we have \be\label{CompClassFun} N({\cal N}(L))={\frak N}(L).\ee

The categories $s(\mbox{\tt Vect})$ and {\tt C}$^+$({\tt Vect}) have well-known monoidal structures (we denote the unit objects by $I_s$ and $I_{\tt C}$, respectively). The normalization functor $N:s(\mbox{{\tt Vect}})\rightarrow
\mbox{\tt C}^+(\mbox{{\tt Vect}})$ is lax monoidal, i.e. it respects the tensor products and unit objects up to coherent chain maps $\ze:I_{{\tt C}}\to N(I_s)$ and
$$EZ_{S,T}:N(S)\otimes N(T)\to N(S\otimes T)$$ (functorial in $S,T\in s(\mbox{\tt Vect})$), where $EZ_{S,T}$ is the
Eilenberg-Zilber map. Functor $N$ is lax comonoidal or oplax monoidal as well, the chain morphism
being here the Alexander-Whitney map $AW_{S,T}.$ These chain maps are inverses of
each other up to chain homotopy, $EZ\;AW=1,\;AW\,EZ\sim 1.$\medskip

The Eilenberg-Zilber map is defined as follows. Let $a\otimes b\in N(S)_p\otimes
N(T)_q\subset S_p\otimes T_q$ be an element of degree $p+q$. The chain map $EZ_{S,T}$ sends $a\otimes b$ to an
element of $N(S\otimes T)_{p+q}\subset (S\otimes
T)_{p+q}=S_{p+q}\otimes T_{p+q}$. We have $$EZ_{S,T}(a\otimes
b)=\sum_{(p,q)-\mbox{shuffles }(\zm,\zn)}\op{sign}(\zm,\zn)\;\;s_{\zn_q}(\ldots
(s_{\zn_1}a))\,\otimes\, s_{\zm_p}(\ldots (s_{\zm_1}b))\in
S_{p+q}\otimes T_{p+q},$$ where the shuffles are permutations of
$(0,\ldots, p+q-1)$ and where the $s_i$ are the degeneracy operators.

\subsection{Monoidal structure and obstruction}

We now come back to the construction of a chain map $\zf:{\frak
N}(L)\otimes {\frak N}(L)\to {\frak N}(L\boxtimes L)$.\medskip

For $L'=L$, the linear map (\ref{NonFun}) reads
$$\ell_2: ({\cal N}(L)\otimes {\cal N}(L))_2\ni
(v,w)\otimes(v',w')\mapsto (v\otimes v',w\otimes w')\in
{\cal N}(L\boxtimes L)_2.$$ If its obvious extensions $\ell_n$ to all other spaces $({\cal N}(L)\otimes {\cal N}(L))_n$ define a simplicial map $\ell:{\cal N}(L)\otimes {\cal N}(L)\to {\cal N}(L\boxtimes L),$ then $$N(\ell):N({\cal N}(L)\otimes {\cal N}(L))\to N({\cal N}(L\boxtimes L))$$ is a chain map. Its composition with the Eilenberg-Zilber chain map $$EZ_{{\cal N}(L),{\cal N}(L)}:N({\cal N}(L))\otimes N({\cal N}(L))\to N({\cal N}(L)\otimes {\cal N}(L))$$ finally provides the searched chain map $\zf$, see Equation (\ref{CompClassFun}).

However, the $\ell_n$ do not commute with all degeneracy and face operators. Indeed, we have for instance $$\ell_2((d_2^3\otimes d_2^3)((u,v,w)\otimes(u',v',w')))=(u\otimes u',(v\circ w)\otimes (v'\circ w')),$$ whereas $$d_2^3(\ell_3((u,v,w)\otimes(u',v',w')))=(u\otimes u',(v\otimes v')\circ(w\otimes w')).$$ Equation (\ref{NonFunUP}), which means that $\boxtimes:L\times L'\to L\boxtimes L'$ is not a functor, shows that these results do not coincide.\medskip

A natural idea would be to change the involved monoidal structures
$\boxtimes$ of {\tt VectCat} or $\otimes$ of {\tt C}$^{+}$({\tt
Vect}). However, even if we substitute the Loday-Pirashvili tensor
product $\otimes_{\op{LP}}$ of 2-term chain complexes of vector
spaces, i.e. of linear maps \cite{LP98}, for the usual tensor
product $\otimes$, we do not get ${\frak N}(L)\otimes_{\op{LP}}
{\frak N}(L)={\frak N}(L\boxtimes L).$\bigskip\bigskip

\noindent{\bf Acknowledgements.} The authors thank the referee for
having pointed out to them additional relevant literature.

\def\cprime{$'$} \def\cprime{$'$} \def\cprime{$'$} \def\cprime{$'$}

\end{document}